\documentclass{amsart}[12pt]
\usepackage[hidelinks]{hyperref}
\usepackage{verbatim}
\usepackage{yfonts} %
\usepackage{amssymb} %
\usepackage{amsthm}
\usepackage{array}
\usepackage{booktabs}%
\usepackage{hhline}%
\usepackage{xy} %
\usepackage{epsfig}%
\usepackage{color}%
\usepackage{upgreek}
\usepackage[english]{babel}
\usepackage{epigraph}%
\usepackage{fancybox}%
\usepackage{shadow}
\usepackage{afterpage}
\usepackage{mathrsfs}
\usepackage{enumitem}
\usepackage{tabularx}
\usepackage{subcaption}
\usepackage{graphicx}
\usepackage{type1cm}
\usepackage{eso-pic}
\usepackage{color}
\usepackage{upgreek}
\usepackage[foot]{amsaddr}
\usepackage{bigints}

\newtheorem{theorem}{Theorem}

\theoremstyle{definition}
\newtheorem{definition}{Definition}
\newtheorem{remark}{Remark}

\theoremstyle{plain}

\newtheorem*{hadamard}{Hadamard Theorem}

\newtheorem*{C-V}{Cohn-Vossen Theorem}
\newtheorem*{dC-W}{do Carmo--Warner Theorem}

\newcommand{\vt}{\vspace{.1cm}}

\newcommand{\R}{\mathbb{R} }

\newcommand{\h}{\mathbb{H}}

\newcommand{\s}{\mathbb{S}}

\renewcommand{\rho}{\varrho}
\renewcommand{\theta}{\varTheta}
\renewcommand{\Theta}{\varTheta}
\renewcommand{\Sigma}{\varSigma}
\renewcommand{\Omega}{\varOmega}
\renewcommand{\Lambda}{\varLambda}
\renewcommand{\tau}{\uptau}
\usepackage{amsmath}

\newcommand{\overbar}[1]{\mkern 1.5mu\overline{\mkern-1.5mu#1\mkern-1.5mu}\mkern 1.5mu}

\makeatletter
\newcommand{\tpitchfork}{%
  \vbox{
    \baselineskip\z@skip
    \lineskip-.52ex
    \lineskiplimit\maxdimen
    \m@th
    \ialign{##\crcr\hidewidth\smash{$-$}\hidewidth\crcr$\pitchfork$\crcr}
  }%
}
\makeatother

\begin{document}

\title[Convex hypersurfaces of Riemannian manifolds]
{A survey on convex hypersurfaces \\ of Riemannian manifolds}
\author{Ronaldo Freire de Lima}
\address[A1]{Departamento de Matem\'atica - Universidade Federal do Rio Grande do Norte}
\email{ronaldo.freire@ufrn.br}
\maketitle

\smallskip
\begin{center}
{\emph{Dedicated to Professor Renato Tribuzy\\ on the occasion of his 75th birthday}}
\end{center}

\begin{abstract}
 We survey the main extensions of the classical Hadamard, Liebmann and Cohn-Vossen
rigidity theorems on convex surfaces of $3$-Euclidean space
to the context of convex hypersurfaces
of Riemannian manifolds. The results we present include the one  by
Professor Renato Tribuzy (in collaboration with H. Rosenberg) on rigidity of
convex surfaces of homogeneous $3$-manifolds.
\end{abstract}

\section{Introduction}
Convexity is an ancient  and fundamental geometric concept attributed to
subsets of spaces. It was first considered by Archimedes in his celebrated book:
\emph{On the sphere and cylinder}, published in 225 B.C., which curiously
has two of the most classical convex surfaces in its very title.

In  Euclidean space $\R^3,$ a nonempty subset $\Omega$ is called
\emph{convex} if it contains the line segment joining any two of its points.
A distinguished property of  convex sets
is that they can be characterized by the geometry of their boundaries.
More precisely, a closed proper subset $\Omega$ in $\R^3$ is convex if and only if
there exists at least one \emph{supporting plane} $\Pi$  at any point
$p$ on its boundary $\partial\Omega$,  meaning that
$\Omega$ lies in one of the half-spaces determined by
$\Pi\owns p$ (see  \cite{minkowski}, pp. 137--140).

Due to  this characterization of convex sets,
a smooth embedded surface $\Sigma$ of $\R^3$ is said to be \emph{convex} if it is the boundary
of  an open convex set $\Omega$ in $\R^3,$ which is then called a \emph{convex body}.
Planes, spheres, and right circular cylinders are canonical examples of convex surfaces.
The \emph{saddle}, graph of the function $z=x^2-y^2,$ is a typical example
of an embedded non convex surface of $\R^3.$

Clearly, if $\Sigma=\partial\Omega$ is a convex surface, the
supporting planes of $\Omega$ are precisely  the tangent planes of $\Sigma.$
In particular, for any $p\in\Sigma,$  a local graph $\Sigma'\subset\Sigma$ through $p$
--- defined over an open set of $T_p\Sigma$ --- is
necessarily contained in one of the closed half-spaces determined by $T_p\Sigma.$ This property,
being local, can also  be attributed to immersed surfaces (possibly  having self-intersections)
and is called \emph{local convexity}.

The formula for the Gaussian curvature of a graph immediately
gives that a locally convex surface $\Sigma$ has nonnegative Gaussian curvature.
(The converse does not hold, as we shall show in the next section.) In addition,
the Gaussian curvature of $\Sigma$ at a point $p$ is positive if and only if the local graph
$\Sigma',$ as described above, intersects $T_p\Sigma$ only at $p.$ If so, we say that
$p$ is an \emph{elliptic point}, and also that  $\Sigma$ is \emph{strictly convex} at $p.$
Thusly, spheres are strictly convex everywhere, whereas
right circular cylinders are nowhere strictly convex.

In \cite{hadamard}, J. Hadamard established the striking fact that
compact locally strictly convex surfaces in $\R^3$ are necessarily embedded, convex,
and diffeomorphic to the unit sphere $\s^2.$ The precise statement is as follows.

\begin{hadamard}
Let $\Sigma$ be a compact connected smooth surface  immersed  in Euclidean space $\R^3$
with positive Gaussian curvature.
Then, $\Sigma$ is orientable, embedded, and convex. In addition, its Gauss map $N:\Sigma\rightarrow\s^2$ is
a diffeomorphism.
\end{hadamard}

J. Stoker \cite{stoker} extended Hadamard Theorem by proving that, if $\Sigma$ is a complete
noncompact immersed surface in $\R^3$ with positive Gaussian curvature,
then it is a graph over a convex open set in $\R^2.$ Usually, these two results are put
together and then called the Hadamard--Stoker Theorem.

A compact connected surface of $\R^3$ with positive Gaussian curvature is called
an \emph{ovaloid}. So, Hadamard Theorem tells us that ovaloids are embedded, convex and
diffeomorphic to $\s^2.$ Considering this fact, it is natural to ask on which conditions an
ovaloid is necessarily a round (i.e., totally umbilical) sphere of $\R^3.$

H. Liebmann \cite{liebmann} addressed this question and provided an answer by
proving that an ovaloid with either constant mean curvature or constant
Gaussian curvature is a totally umbilical  sphere. In the  mean curvature case, a
less prestigious -- although stronger --
result was obtained earlier by J. Jellett \cite{jellett}, who assumed the surface to be star shaped, instead
of an ovaloid. For this reason, this Liebmann Theorem has been
frequently  referred to as  the Liebmann--Jellett Theorem. In the
Gaussian curvature case, a  proof given by D. Hilbert (cf. \cite[Appendix 5]{hilbert}) has imposed itself
along the time, so that the result bears his name together with Liebmann's.

The deepest theorem regarding ovaloids is probably the one due to S. Cohn-Vossen \cite{cohn-vossen}, which
asserts that such a surface is \emph{rigid}. This means that, if $\Sigma_1$ and $\Sigma_2$ are isometric
ovaloids, then they are congruent, that is, they coincide up to a rigid motion of $\R^3.$
In fact, Cohn-Vossen proved
his theorem assuming $\Sigma_1$ and $\Sigma_2$ analytic. Afterwards, G. Herglotz \cite{herglotz}
provided a  succinct alternate proof under the  weaker
assumption that  $\Sigma_1$ and $\Sigma_2$ are of class $C^2.$

Since their appearance to date, these classical  rigidity theorems on convex surfaces
by Hadamard, Liebmann, and Cohn-Vossen have been extended, in  many ways,
to the more general context of hypersurfaces in  Riemannian manifolds.
In what follows, we survey the main results obtained on this matter, which
have a significant contribution of  Brazilian mathematicians --- most notably
Manfredo do Carmo --- as we shall see.

We will also take this opportunity to present, in our last section,
a result by
Professor Renato Tribuzy in collaboration with H. Rosenberg on rigidity of
convex surfaces of homogeneous $3$-manifolds.




\section{Preliminaries} \label{sec-preliminaries}

Throughout  the paper, unless otherwise stated, the  Riemannian manifolds
we consider are all orientable, smooth (of class $C^\infty$), and of dimension
at least $2.$
Given a Riemannian manifold $\overbar M^{n+1},$ we will consider its hypersurfaces
mostly as isometric immersions
$f:M^n\rightarrow\overbar M^{n+1},$
where $M^n$ is some $n$-dimensional
Riemannian manifold.

Occasionally,  we shall consider isometric immersions
$f:M^n\rightarrow\overbar M^{n+p}$
of codimension $p\ge 1.$ In this setting,
we will  write $TM$ and $TM^\perp$ for the tangent bundle and normal
bundle of $f$, respectively, and
\[
\alpha_f:TM\times TM\rightarrow TM^\perp
\]
for its second fundamental form, that is,
\[
\alpha_f(X,Y)=\overbar{\nabla}_XY-\nabla_XY,
\]
where $\overbar{\nabla}$ and $\nabla$ denote the Riemannian connections of $\overbar{M}$ and $M,$
respectively. Also, given $\xi\in TM^\perp,$ we define  $A_\xi:TM\rightarrow TM$ by
\[
A_\xi X=-(\text{tangential component of} \,\,\, \overbar{\nabla}_X\xi)
\]
and call it the \emph{shape operator} of $f$ in the normal direction $\xi.$

As is well known, $A_\xi$ is self-adjoint, so that, for each $x\in M,$
there exists an  orthonormal basis $\{X_1,\dots, X_n\}\subset T_x M$ of eigenvectors
of $A_\xi.$ Each vector $X_i$ is called a \emph{principal direction} of $f$ at $x$ (with respect to the normal $\xi$),
and the corresponding eigenvalue $\lambda_i$ is called the \emph{principal curvature} of $f$ at $x$ (with respect to the normal $\xi$).
So, we have
\[
A_\xi X_i=\lambda_i X_i, \,\,\,\, i=1,\dots, n.
\]

The following equality relating the second fundamental form $\alpha_f$ and a shape operator $A_\xi$ holds:
\[
\langle\alpha_f(X,Y),\xi\rangle=\langle A_\xi X,Y\rangle \,\,\,\, \forall \,X,Y\in TM, \,\, \xi\in TM^\perp,
\]
where $\langle \,,\, \rangle$ stands for the Riemannian metric of both $M$ and $\overbar{M}.$

The second fundamental form $\alpha_f$ of an isometric immersion
$f:M^n\rightarrow\overbar M^{n+p}$ is said to be \emph{semi-definite} if, for all $\xi\in TM^\perp,$
the 2-form
\[
(X,Y)\in TM\times TM\mapsto \langle\alpha_f(X,Y),\xi\rangle
\]
is semi-definite (i.e., the nonzero eigenvalues of the shape operator
$A_\xi$  have all the same sign) on $M.$
Also, we say that $\alpha_f$ is \emph{positive}
(respect. \emph{negative}) \emph{semi-definite} in the
normal direction $\xi\in TM^\perp$ if the nonzero eigenvalues
of $A_\xi$ are all positive (respect. {negative}) on $M.$

We remark that, in the above setting, the semi-definiteness of
the second fundamental form $\alpha_f$ does not
imply that it is either positive or negative semi-definite in a given
normal direction $\xi\in TM^\perp.$  To see that in the case of hypersurfaces,
consider a smooth plane curve
$\gamma: I\subset\R\rightarrow\R^2,$ 
and define the \emph{cylinder over} $\gamma$ as the immersion
\[
\begin{array}{cccc}
f\colon & I\times\R^{n-1} & \rightarrow &\R^2\times\R^{n-1}\\
        & (t,x)     & \mapsto     &(\gamma(t),x),
\end{array}
\]
which is easily seen to be  an orientable hypersurface in $\R^{n+1}.$ Also,
for any of the two choices of a unit $\xi\in T(I\times\R)^\perp,$ one has
\[
A_\xi=\pm
\left[
\begin{array}{ccccc}
k & & & &\\
  & 0 & & &\\
  &  &\ddots & &\\
  &  & & &0
\end{array}
\right],
\]
where $k$ is the curvature of $\gamma.$ In particular,
$\alpha_f$ has semi-definite second fundamental form. In addition, assuming
that there exist  $t_{1}, t_2\in I$ satisfying $k(t_{1})<0<k(t_2),$
we have that, in  any  normal direction
$\xi\in T(I\times\R)^\perp,$
$\alpha_f$ is positive semi-definite at $(t_1,x)$ if and only if
it is negative semi-definite at $(t_2,x).$ Therefore, $\alpha_f$ is
neither positive semi-definite nor negative semi-definite in any normal direction (Fig. \ref{fig-cylinder}).

\begin{figure}[htbp]
\begin{center}
\includegraphics{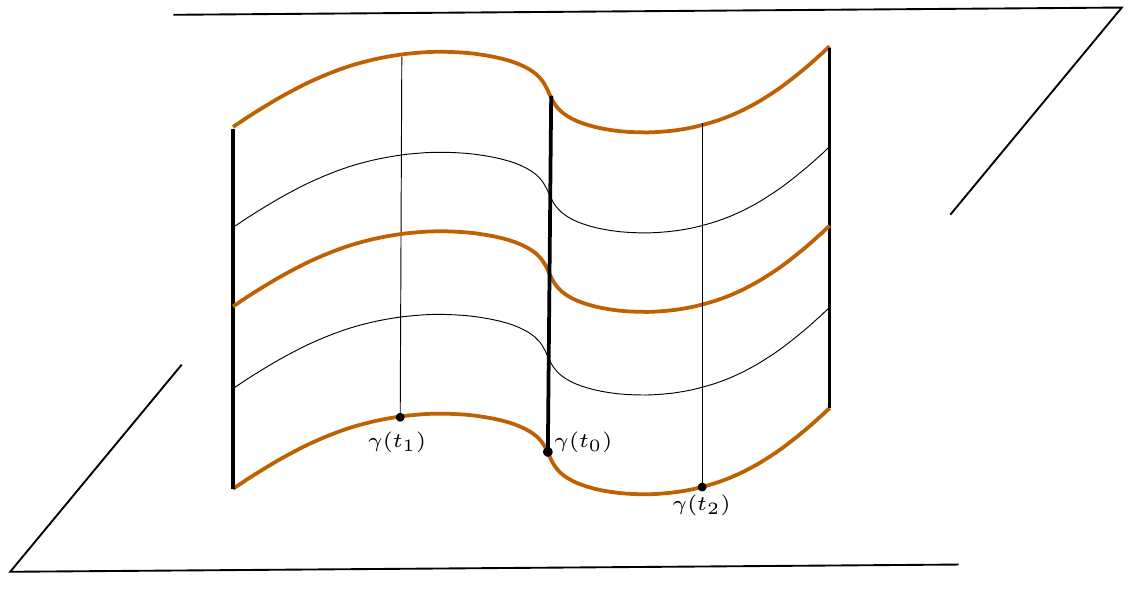}
\caption{\small Non locally convex cylinder over a plane curve $\gamma.$}
\label{fig-cylinder}
\end{center}
\end{figure}

It should also be noticed that, for $n=2,$ $\Sigma:=f(I\times\R)$ is not
locally convex (as we discussed in the introduction) at the points $f(t_0,x)\in\Sigma$ such that $k(t_0)=0,$
and that the Gaussian curvature $\Sigma$ vanishes identically
(since the same is true for one of its principal curvatures).
Thus, in $\R^3,$  nonnegativity of Gaussian curvature does not imply local convexity.

It is time for us to settle the concept of local convexity for a general orientable hypersurface
$f\colon M^n\rightarrow\overbar M^{n+1},$  and establish its relations with the second
fundamental form $\alpha_f.$  With this purpose, let us first observe that,
for  any point $x\in M,$ there is an open neighborhood $U$ of $x$ in $M$ such that
$\Sigma:=f(U)$ is a graph of a function $\phi$
(under the normal exponential map of $\overbar M^{n+1}$) over an open set $V\owns 0$ of
the tangent space $f_*(T_xM).$ In this context, $f$ is said to be \emph{locally convex}
(resp. \emph{locally strictly convex}) at $x$
if, for a suitable normal direction $\xi\in TM^\perp,$ the function
$\phi$ is nonnegative on $V$ (resp. positive on $V-\{0\}).$

A hypersurface $f\colon M^n\rightarrow\overbar M^{n+1}$ is then  called \emph{locally convex}
(resp. \emph{locally strictly convex}) if it is {locally convex}
(resp. {locally strictly convex}) at any point of $M.$

Following a suggestion by M. do Carmo, R. Bishop showed in \cite{bishop} that
local convexity is equivalent to semi-positiveness of the second fundamental form,
as stated below.

\begin{theorem}[Bishop \cite{bishop}]
An orientable  hypersurface $f\colon M^n\rightarrow\overbar M^{n+1}$
is locally convex (resp. locally  strictly convex) if and only if,
in a suitable normal direction $\xi\in TM^\perp,$
the second fundamental form $\alpha_f$ is positive semi-definite (resp. positive definite).
\end{theorem}

We point out that the concepts of convex set and convex body extend naturally
to subsets of any totally convex Riemannian manifold $\overbar M^{n+1}.$ (Recall that
a Riemannian manifold is called \emph{totally convex} if there exists a unique geodesic
joining any two of its points.)

\begin{remark} \label{rem-convexbody}
By adapting an argument by J. Heijenoort \cite{heijenoort},
it is possible to show that, in any Hadamard manifold
(i.e., complete simply connected Riemannian manifold with nonnegative sectional curvature) $\overbar M^{n+1},$
a complete embedded hypersurface $f:M^n\rightarrow\overbar M^{n+1}$ is locally convex if and only if
$f(M)$ is the boundary of a convex body. On the other hand, as pointed out in \cite[Remark 3]{alexander},
there exist simply connected Riemannian manifolds whose
geodesic spheres are embedded and locally convex, yet they do
not bound  convex bodies.
\end{remark}

We will denote by $\mathbb{Q}_\epsilon^{n+1}$
the $(n+1)$-dimensional simply connected space form of constant sectional
curvature $\epsilon\in\{0,1,-1\},$ i.e.,
the Euclidean space $\R^{n+1}$ ($\epsilon=0$),
the unit sphere $\s^{n+1}$ ($\epsilon =1$), and the hyperbolic space $\h^{n+1}$ ($\epsilon=-1$).

Given an oriented hypersurface $f:M\rightarrow\mathbb{Q}_\epsilon^{n+1},$ let us see how
the sectional curvature of $M$ affects the semi-definiteness of the second fundamental form
$\alpha_f.$ To this end, choose an orthonormal frame
$\{X_1, \dots, X_n\}\subset TM$  of principal directions of $f$ with
corresponding principal curvatures $\lambda_1, \dots ,\lambda_n.$
Denoting by $K_M$ the sectional curvature of $M,$ the well known
Gauss equation for hypersurfaces of space forms yields (see, e. g., \cite{dajczer})
\[
K_M(X_i, X_j)=\epsilon +\lambda_i\lambda_j \,\,\, \forall i\ne j\in\{1,\dots, n\}.
\]
Therefore, \emph{the second fundamental form of $f:M\rightarrow\mathbb{Q}_\epsilon^{n+1}$
is semi-definite if and only if} $K_M\ge\epsilon.$

We conclude this preliminary section by  introducing the fundamental
concept of rigidity of isometric immersions.
\begin{definition}
  We say that an isometric immersion $f:M^n\rightarrow\overbar M^{n+p}$ is \emph{rigid} if, for any
  other isometric immersion $g:M^n\rightarrow\overbar M^{n+p},$ there exists an ambient isometry
  $\Phi:\overbar M^{n+p}\rightarrow\overbar M^{n+p}$ such that $g=\Phi\circ f.$
\end{definition}

\section{Convex hypersurfaces of $\mathbb{Q}_\epsilon^{n+1}$}

Regarding  extensions of Hadamard, Liebmann and Cohn-Vossen Theorems to
more general ambient manifolds, a first natural
step is to verify  their validity for
hypersurfaces of the space forms $\mathbb{Q}_\epsilon^{n+1}.$

In the Euclidean case, based on
a Hadamard--Stoker type theorem due to J. Heijenoort \cite{heijenoort}, and on a local rigidity theorem due to
R. Beez \cite{beez} and W. Killing \cite{killing}, R. Sacksteder \cite{sacksteder1,sacksteder2} succeeded to
extend both  Hadamard--Stoker and Cohn-Vossen Theorems  to complete
hypersurfaces of $\R^{n+1}$ with semi-definite second fundamental form,
as stated below. We stress the fact that semi-definiteness of
the second fundamental form is a  weaker condition than local convexity, and that in both
Hadamard--Stoker and Cohn-Vossen Theorems the surfaces are assumed to be strictly convex.

\begin{theorem}[Sacksteder \cite{sacksteder1,sacksteder2}]  \label{th-sacksteder}
Let $f\colon M^n\rightarrow\R^{n+1}$ be an orientable, complete, and connected hypersurface with
semi-definite second fundamental form,
which is strictly convex  at one point. Then, the following  hold:
\begin{itemize}
   \item[{\rm i)}] $f$ is an embedding and  $M$ is homeomorphic  to either $\s^n$ or $\R^n.$
   \item[{\rm ii)}] $f(M)$  is the boundary of a convex body in $\R^{n+1}.$
   \item[{\rm iii)}] $f$ is rigid.
\end{itemize}
\end{theorem}

The original proofs of these Sacksteder's results
are  rather involved. In \cite{docarmo-lima}, M. do Carmo and E. Lima
provided a simpler proof of (i)--(ii) by means of  Morse Theory. The main idea consists in
considering the tangent space of $f(M)$ at the strictly convex point $f(x)\in f(M),$
and moving it in the direction of the inner normal $\xi(x)$  until it reaches
a singularity, if any. With this approach, they showed
that either  no singularity occurs, in which case
$M$ is homeomorphic to $\R^n$ (in fact, $f(M)$ is a graph over a convex open set in a
hyperplane orthogonal to $\xi(x)$
in $\R^{n+1}$), or the singular set is a singleton,
in which case $f(M)$ is an embedded topological $n$-sphere of $\R^{n+1}.$
Once  established the embeddedness of $f,$ the proof of (ii) is standard.

For the proof of the rigidity assertion (iii) of Theorem \ref{th-sacksteder},
one can argue as follows.
Since $f$ has semi-definite second fundamental form, Gauss equation gives
that $M$ has nonnegative sectional curvature. Hence,
any hypersurface $g:M^n\rightarrow\R^{n+1}$ has semi-definite second fundamental form.
Assuming $n\ge 3,$ and recalling that $f$ is strictly convex at a point, one has  that
the set $M'\subset M$ on which the second fundamental form $\alpha_f$ of $f$ has rank at least
$3$ is nonempty. In this case, the Beez--Killing Theorem we mentioned asserts that $f|_{M'}$ is rigid.
From this, and the semi-definiteness  of $\alpha_f$ and $\alpha_g,$ one has either
$\alpha_f=\alpha_g$ or $\alpha_f=-\alpha_g$ on $M'.$
The same conclusion holds in the case $n=2$ from a result by Herglotz \cite{herglotz}.
Now, Sacksteder's main result in \cite{sacksteder2} ensures that either of these equalities
extends to the whole of $M.$ The rigidity of $f$, then, follows from the Fundamental
Theorem for hypersurfaces of Euclidean spaces.

The cylinders over curves (see Section \ref{sec-preliminaries}) show that the hypothesis on the
existence of an elliptic point is necessary in Sacksteder Theorem. Let us see that the same is true regarding
completeness. Indeed, the graph of the function
$z=x^3(1+y^2),$ $|y|<1/2,$  is easily seen to be non convex along the line $x=0.$
However, its second
fundamental form is semi-definite everywhere, as can be verified by a direct computation. Therefore,
by Sacksteder Theorem, this graph  cannot be a part of a complete surface with semi-definite second fundamental form.

Inspired by Sacksteder's results, M. do Carmo and F. Warner considered
the analogous problem for compact hypersurfaces  of $\s^{n+1}$ and $\h^{n+1},$ obtaining the following

\begin{theorem}[do Carmo -- Warner \cite{docarmo-warner}] \label{th-docarmo-warner}
Let  $f:M^n\rightarrow\s^{n+1}$  be a non-totally geodesic hypersurface
with semi-definite second fundamental form,
where $M^n$ is a compact,  connected, and orientable
Riemannian manifold.
Then, the following hold:
\begin{itemize}
   \item[{\rm i)}] $f$ is an embedding and  $M$ is homeomorphic  to \,$\s^n.$
   \item[{\rm ii)}] $f(M)$  is the boundary of a convex body contained in an open hemisphere of \,$\s^{n+1}.$
   \item[{\rm iii)}] $f$ is rigid.
\end{itemize}
Moreover, the assertion {\rm (i)} and the convexity property in {\rm (ii)}
still hold if one replaces the sphere $\s^{n+1}$ by the hyperbolic space  \,$\h^{n+1}.$
\end{theorem}

The proof of do Carmo--Warner Theorem
relies on the properties of the so-called \emph{Beltrami maps},  which are  central projections
of  open hemispheres of $\s^{n+1}$ (respect.  hyperbolic space $\h^{n+1}$) on
suitable hyperplanes of Euclidean space $\R^{n+2}$ (respect. Lorentz space $\mathbb{L}^{n+2}$).

More precisely, in the spherical case, the
Beltrami map for the hemisphere $\mathcal H$ of
$\s^{n+1}$ centered at $e_{n+2}:=(0,0,\dots, 0,1)\in\R^{n+2}$ is
\begin{equation} \label{eq-beltramimap}
\begin{array}{cccc}
\varphi: & \mathcal{H} & \rightarrow & \R^{n+1}\\
         &     x         & \mapsto     & \frac{x}{x_{n+2}}\,,
\end{array}
\end{equation}
where $x_{n+2}$ stands for the $(n+2)$-th coordinate of $x$ in $\R^{n+2}$
(Fig. \ref{fig-beltramispherical}).
\begin{figure}[htbp]
\begin{center}
\includegraphics{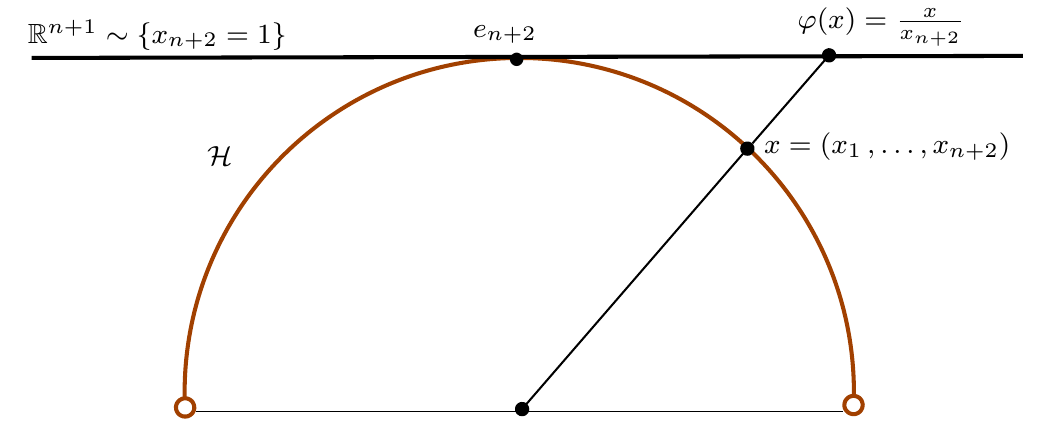}
\caption{\small Beltrami map of the hemisphere $\mathcal H$ centered at $e_{n+2}.$}
\label{fig-beltramispherical}
\end{center}
\end{figure}

Analogously, in the hyperbolic case, the Beltrami map is given by
\begin{equation} \label{eq-beltramimap10}
\begin{array}{cccc}
\varphi: & \mathbb{H}^{n+1} & \rightarrow & B^{n+1}\\
         &     x         & \mapsto     & \frac{x}{x_{n+2}}\,,
\end{array}
\end{equation}
where $ B^{n+1}$ stands for the unit ball of the
affine subspace $x_{n+2}=1$\, of \,$\mathbb{L}^{n+2}$
(Fig. \ref{fig-beltramihyperbolic}).

It is easily verified  that, in both cases, the Beltrami map $\varphi$ is a diffeomorphism. Moreover,
$\varphi$ and its inverse are geodesic maps,
that is, they take geodesics to geodesics and, in particular,
convex sets to convex sets.

\begin{figure}[htbp]
\begin{center}
\includegraphics{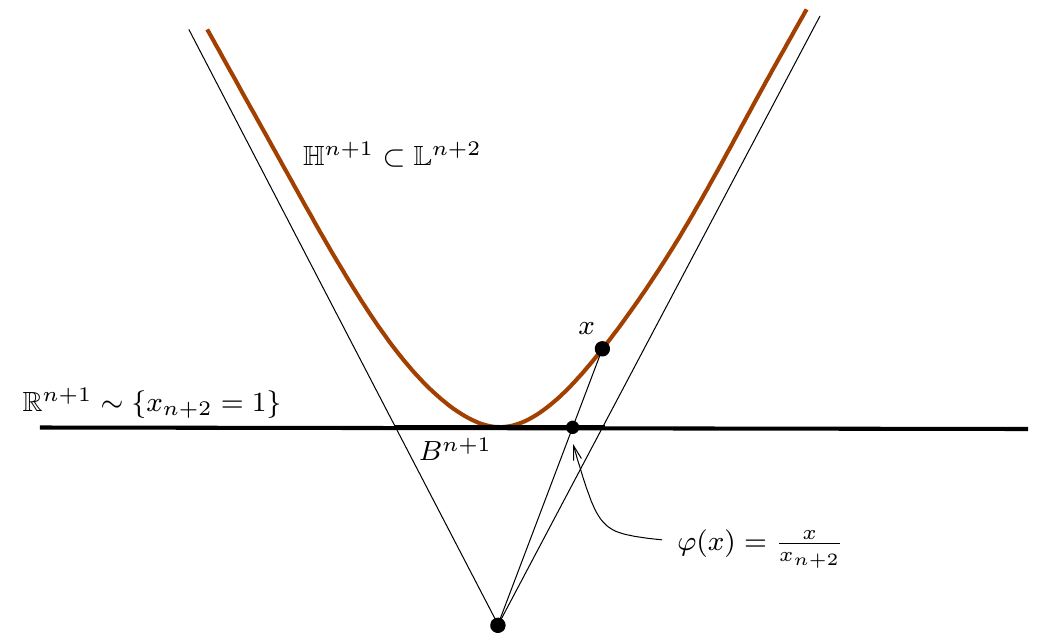}
\caption{\small Beltrami map of hyperbolic space $\h^{n+1}$.}
\label{fig-beltramihyperbolic}
\end{center}
\end{figure}

In their proof, do Carmo and Warner cleverly used Beltrami maps combined with
Sacksteder Theorem \ref{th-sacksteder}  to
show that, in the spherical case, there exists a point $x\in M$ at which
the second fundamental form of $f$ is definite, and also
that $f(M)-\{f(x)\}$ is contained in an open hemisphere $\mathcal H'$ of $\s^{n+1}$
whose boundary meets $f(M)$ at $f(x).$ From this, since $M$ is compact, they concluded
that $f(M)$ is contained in an open  hemisphere $\mathcal H$ of $\s^{n+1}.$ Considering
now both cases, spherical and hyperbolic, and
denoting by $\varphi$ the Beltrami map of either $\mathcal H$ or $\h^{n+1},$ as they show,
$\varphi\circ f:M\rightarrow\R^{n+1}$ is complete and convex.
Then, from Sacksteder Theorem, $\varphi\circ f$ is an embedding and
$\varphi(f(M))$ is the body of a convex body in $\R^{n+1},$
which implies that the same is true for $f$ and  $f(M),$ respectively.

A similar idea, involving  special mixed Beltrami-like maps,
was used for the proof of the rigidity of $f$
in the spherical case.
At the very end of the paper, do Carmo and Warner conjectured that the rigidity of $f$ holds in
the hyperbolic case as well.

In \cite{andrade-delima}, jointly with R. de Andrade, the author extended
do Carmo and Warner theorem to immersions of arbitrary codimension, and also settled
affirmatively their aforementioned conjecture. More precisely, the following result was obtained.

\begin{theorem}[Andrade -- de Lima \cite{andrade-delima}]   \label{th-andrade-delima1}
Let $f:M^n\rightarrow\mathbb Q_\epsilon^{n+p}$ be an
isometric immersion of a  compact connected Riemannian manifold  into
the $(n+p)$-dimensional space form of constant sectional curvature $\epsilon=\pm 1.$
Assume  that $f$ is non-totally geodesic and has
semi-definite second fundamental form.
Under these conditions, the following assertions hold:
\begin{itemize}
  \item[\rm i)] $f$ is an embedding of $M$ into a totally geodesic $(n+1)$-dimensional
submanifold $\mathbb Q_\epsilon^{n+1}\subset \mathbb Q_\epsilon^{n+p}.$
  \item[\rm ii)] $f(M)$ is the boundary of a compact convex set of $\mathbb Q_\epsilon^{n+1}.$
  In particular, $M$ is homeomorphic to $\s^n.$
  \item[\rm iii)] $f$ is rigid.
\end{itemize}
\end{theorem}

In the proof of Theorem \ref{th-andrade-delima1}, as in do Carmo--Warner Theorem, we
make use of the Beltrami maps $\varphi$ to recover the Euclidean situation. A fundamental result, which allowed us to
apply this method, is the following relation we obtained between the second fundamental forms
$\alpha_f$ and $\alpha_g$ of $f$ and $g:=\varphi\circ f\colon M\rightarrow\R^{n+1}:$
\begin{equation}\label{eq-sff}
\langle\alpha_f(f_*X,f_*Y),\xi_f\rangle=\phi^2\langle\alpha_g(g_*X,g_*Y),\xi_g\rangle,  \,\,\, X, Y\in T\mathcal H^{n+p},
\end{equation}
where $\xi_f$ and $\xi_g$ are (bijectively related)  normal fields to $f$ and $g,$ respectively,
$\mathcal{H}^{n+p}$ denotes either an open hemisphere of $\s^{n+p}\subset\R^{n+p+1}$ or
$\h^{n+p}\subset\mathbb L^{n+p+1}$ (with standard metric $\langle\,,\,\rangle$),
and $\phi$ stands for the function
\[
x=(x_1\,, \dots, x_{n+p+1})\in \mathcal{H}^{n+p} \,\mapsto\, \phi(x):= 1/x_{n+p+1}\,.
\]
It should be noted that no such relation was established by do Carmo and Warner for
hypersurfaces, i.e., for $p=1.$

In the spherical case,  by means of a result due to Dajczer and Gromoll \cite{dajczer-gromoll}, we show that there is a point
$x\in M$ at which the second fundamental form of $f$ is positive definite. Then,
considering the  Beltrami map for the hemisphere of $\s^{n+p}$ centered at $x,$ together with
a reduction of codimension theorem for hypersurfaces of Euclidean space, due to Jonker \cite{jonker}, we show (i) and (ii) for
$\epsilon=1.$ The case $\epsilon=-1$ is analogous.

As for the rigidity of $f,$  we saw that the spherical case was settled by do Carmo and Warner.
For the hyperbolic case, we again use Beltrami maps and identity \eqref{eq-sff}
to show that the set of totally umbilical points of $f$ does not disconnect $M.$ In this case, for $n>2,$
the rigidity of $f$ follows from another celebrated rigidity theorem by Sacksteder \cite{sacksteder2}
(see \cite[Theorem 6.14]{dajczer} for an alternate proof). The case
$n=2$ was proved by V. Fomenko and G. Gajubovin (cf. \cite[Theorem 5]{fomenko-gajubov}).

The  technique  of combining Beltrami maps with equality \eqref{eq-sff} also gives  the following
result, whose first part constitutes a Hadamard-type theorem. (The final statement follows from Theorem \ref{th-andrade-delima1}.)

\begin{theorem}[Andrade -- de Lima \cite{andrade-delima}]  \label{th-andrade-delima}
Let $f:M^n\rightarrow\mathcal{H}^{n+1}$ be a  compact connected hypersurface, where $\mathcal{H}^{n+1}$ is either
the open hemisphere of $\mathbb S^{n+1}\subset\R^{n+2}$ centered at $e_{n+2}:=(0,0,\dots 0,1)$
or the hyperbolic space $\h^{n+1}\subset\mathbb{L}^{n+2}.$
Then, the following assertions are equivalent:
\begin{itemize}
\item[\rm i)] $f$ is locally strictly convex.

\item[\rm ii)] The Gauss-Kronecker curvature of $f$ is nowhere vanishing.

\item[\rm iii)] $M$ is orientable and, for a unit normal field $\xi$  on $M,$
the map
$$
\begin{array}{cccc}
\psi: & M^n & \rightarrow & \mathbb S^n \\

      & x   & \mapsto     &  \frac{\xi(x)-\langle \xi(x),e_{n+2}\rangle e_{n+2}}{\sqrt{1-\langle \xi(x),e_{n+2}\rangle^2}}
  \end{array}
$$
is a well-defined diffeomorphism, where  
$\mathbb S^n$ stands for the $n$-di\-men\-sion\-al unit sphere of
the Euclidean orthogonal complement  of $e_{n+2}$ in $\R^{n+2}.$
\end{itemize}
Furthermore, any of the above conditions implies that $f$ is rigid and embeds $M$ onto the
boundary of a compact convex set in $\mathcal{H}^{n+1}.$
\end{theorem}

A natural question arises from do Carmo--Warner Theorem: Could one replace
compact by complete in its statement and still get to the same conclusions?
In the spherical case, the question is irrelevant, since
any sectional curvature of a convex hypersurface $f:M\rightarrow\s^{n+1}$ is at least $1,$ by Gauss equation.
Therefore, by Bonnet--Myers Theorem, $M$ is necessarily compact if the induced metric is complete.
In the hyperbolic case, it is known that there exist complete convex hypersurfaces
of $\h^{n+1}$ which are not embedded (see \cite{spivak}, pg. 124).
So, in this case, the answer for the above question is negative.

In \cite{currier}, J. Currier considered strictly convex hypersurfaces of $\h^{n+1}$
whose principal curvatures are at least $1.$ At points where all the principal curvatures
are greater than $1,$
these  hypersurfaces are locally supported by horospheres
of $\h^{n+1}.$ This means that, at such a point, call it $x,$
the hypersurface is tangent to a horosphere $\Sigma$ of $\h^{n+1}$
and, except for $x$ itself, a neighborhood of $x$
in the hypersurface lies in the convex connected
component of $\h^{n+1}-\Sigma.$
This geometric property allowed Currier to apply Morse Theory in the same way  do Carmo and Lima
did in \cite{docarmo-lima}, showing that such a hypersurface is necessarily embedded and rigid, being either
a topological $n$-sphere bounding a convex body or a horosphere itself.

Again by using Beltrami maps, we were able to obtain in \cite{andrade-delima} the following
version of Currier's result in arbitrary codimension, as stated below.

\begin{theorem}[Andrade -- de Lima \cite{andrade-delima}] \label{th-hconvexity}
Let $f:M^n\rightarrow\h^{n+p}$ be an isometric immersion
of an orientable complete connected Riemannian manifold $M^n$ into the hyperbolic space $\h^{n+p}.$
Assume that there is an orthonormal frame $\{\xi_1\,, \dots ,\xi_p\}$ in $TM^\perp$ such that
all the eigenvalues of the shape operators $A_{\xi_i}$  are at least $1.$
Then, $f$ admits a reduction of codimension to one, $f:M^n\rightarrow\h^{n+1}.$
As a consequence, the following assertions hold:
\begin{itemize}
  \item[\rm i)] $f$ is and embedding and $f(M)$ is the boundary of a convex body in $\h^{n+1}.$
  \item[\rm ii)] $M$ is homeomorphic to either $\s^n$ or $\R^n.$
  \item[\rm iii)] $f$ is rigid and $f(M)$ is a horosphere if $M$ is noncompact.
\end{itemize}
\end{theorem}

Let us consider now extensions of  Liebmann's  theorems to
$\mathbb Q_\epsilon^{n+1}.$

In a series of outstanding papers published between 1956 and 1962,
A. Alexandrov studied constant mean curvature hypersurfaces of Euclidean space  $\R^{n+1},$
establishing the fundamental result that such an embedded compact hypersurface
is a totally umbilical $n$-sphere. The ingenious method employed in his proof,
now called \emph{Alexandrov reflection technique}, became one of the most powerful tools
in the study of a large class of hypersurfaces of Riemannian manifolds having
suitable groups of isometries.
(We shall have a glimpse of this phenomenon in the next section.)

Given a constant mean curvature  embedding $f:M^n\rightarrow\R^{n+1},$ with $M$ compact,
the Alexandrov reflection technique consists in
considering an arbitrary  hyperplane $\Pi$ disjoint from
$\Sigma:=f(M),$ which moves towards it until it becomes tangent
at a point $p\in\Sigma.$ Then, it is shown that there exists
a hyperplane  parallel to $T_p\Sigma$ with respect to which
$\Sigma$ is symmetric. Since $\Pi$ is arbitrary,
$\Sigma$ must be a round sphere.

Let us recall that, given an integer $k\in\{0,1,\dots ,n\},$ the $k$-th mean curvature
of a hypersurface $f:M^n\rightarrow\overbar M^{n+1}$ 
is the function
\[
H_k:={{n}\choose{k}}^{-1}\sigma_k(\lambda_1,\dots,\lambda_n),
\]
where $\lambda_1,\dots, \lambda_n$ are the principal curvatures of $f,$ and $\sigma_k$
is defined as
\[
\sigma_k(\lambda_1,\dots,\lambda_n):=\left\{
\begin{array}{cl}
  1 & {\rm if} \,\, k=0. \\ [1ex]
  \displaystyle\sum_{i_1<\cdots <i_r}\lambda_{i_1}\dots \lambda_{i_r} & {\rm if} \,\, 1\le k\le n.
\end{array}
\right.
\]
In particular, $H_1$ is the mean curvature and $H_n$ is  the Gauss-Kronecker curvature.
When $n = 2,$ $H_2 = k_1 k_2$ is also known as the extrinsic curvature or
Gaussian curvature. We will use this last terminology for the sequence
of the paper in cases where $n = 2.$

\begin{definition}
When the $k$-th mean curvature of $f:M^n\rightarrow\overbar M^{n+1}$
is a constant $H_k,$  we say that $f$ is an $H_k$-\emph{hypersurface}.
\end{definition}

In \cite{korevaar}, N. Korevaar showed that
the Alexandrov reflection technique works for $H_k$-hypersurfaces
in hyperbolic space $\h^{n+1}$ or in open hemispheres of $\s^{n+1}.$
By combining his results with Sacksteder and do Carmo-Warner Theorems,
one easily extends both Liebmann's theorems to $\mathbb Q_\epsilon^{n+1}$ as follows
(see also \cite{montiel-ros}).

\begin{theorem}[Generalized Liebmann Theorem] \label{th-generalizedliebmann}
Let $M^n$ be a compact orientable Riemannian manifold. If
$f:M\rightarrow\mathbb Q_\epsilon^{n+1}$ is a locally strictly convex $H_k$-hypersurface
for some $k\in\{1,\dots,n\},$ then $f(M)$ is a  geodesic sphere
of \,$\mathbb Q_\epsilon^{n+1}.$
\end{theorem}

In Euclidean space $\R^3,$ the Liebmann's theorems can also be proved
by means of the celebrated  \emph{Minkowski formulas} (see, e.g., \cite{montiel-ros2}).
In fact, this method can be adapted for proving Theorem \ref{th-generalizedliebmann}.
Since we have no reference for such a proof,  we shall present it here.

We will make use of the  well known
fact that the mean curvature functions $H_k$ satisfy the inequality
(see, e.g.,  \cite{hardyetal} pg. 52)
\begin{equation}\label{eq-hk}
H_k^{1/k}\ge H_{k+1}^{1/(k+1)} \,\,\, \forall k\in\{1,\dots ,n-1\},
\end{equation}
and that the equality occurs if and only if $\lambda_1=\lambda_2=\ldots =\lambda_n.$

Considering polar coordinates $(r,\theta)$ in $\mathbb Q_\epsilon^{n+1},$ define
the \emph{position vector} $P=s_\epsilon(r)\partial_r$\,, where
\[
s_\epsilon(r):=
\left\{
\begin{array}{lcl}
  \sin r & \text{if} & \epsilon =1. \\[1ex]
  r &  \text{if} & \epsilon=0. \\[1ex]
  \sinh r & \text{if} & \epsilon=-1.
\end{array}
\right.
\]
Also, for a
given compact oriented  hypersurface $f:M\rightarrow\mathbb Q_\epsilon^{n+1}$  with
inward unit normal $\xi,$  define the \emph{support function} $\mu=\langle P, \xi\rangle.$
In this setting, the following
Minkowski identity holds (cf. \cite[Theorem 3.2]{albuquerque}):
\begin{equation}\label{eq-minkowski}
\int_M(c_\epsilon H_{k-1}+\mu H_{k})dM=0, \,\,\, k=1,\dots ,n,
\end{equation}
where $c_\epsilon=c_\epsilon(r):=s_\epsilon'(r),$ and $dM$ stands for
the volume element of $M.$
(The identity \eqref{eq-minkowski} differs from the one in \cite[Theorem 3.2]{albuquerque},
where the minus sign replaces the plus sign. This is due
to the fact that the definition of shape operator in \cite{albuquerque} differs from ours by a sign.
More precisely, in \cite{albuquerque}, the shape operator $A_\xi$ is defined as $A_\xi X=\overbar\nabla_X\xi,$ whereas we define it
as $A_\xi X=-\overbar\nabla_X\xi.$)

\begin{proof}[Proof of Theorem \ref{th-generalizedliebmann}]
  By Sacksteder and do Carmo--Warner Theorems, $f$ is an embedding and
  $M$ is homeomorphic to $\s^n.$
  We can assume $f$ with the inward orientation, so that the support function
  $\mu$ is negative on $M.$
  Also, for $\epsilon=1,$
  $f(M)$ is in an open hemisphere of $\s^{n+1},$ which implies that, for  suitable polar coordinates
  in $\s^{n+1},$ one has $c_1(r)=\cos r>0$ on $f(M).$ For $\epsilon\in\{0,-1\},$
  it is clear that $c_\epsilon >0.$

  Let us consider first the case where $f$ is an $H$-hypersurface. Considering
  the Minkowski equalities for $k=1$ (multiplied by the constant $H>0$) and
  for $k=2,$ we have that
  \[
  \int_M(c_\epsilon H+\mu H^2)dM=0 \quad\text{and}\quad \int_M(c_\epsilon H+\mu H_{2})dM=0.
  \]
  Subtracting these equalities, we get
  \[
  \int_M\mu (H^2-H_2)dM=0.
  \]

  By \eqref{eq-hk}, the integrand in this last integral is non positive. Therefore,
  $H^2=H_2$ on $M,$  which implies that $f$ is totally umbilical, i.e.,  $f(M)$ is a geodesic sphere of
  $\mathbb Q_\epsilon^{n+1}.$

  Now, assume that $f$ is an $H_k$-hypersurface for
  $1<k\le n.$ Multiplying \eqref{eq-minkowski} by $1/H_k^{(k-1)/k}$ and considering that equality also
  for $k=1,$ we get
  \[
  \int_M(c_\epsilon +\mu H)dM=0 \quad\text{and}\quad \int_M\left(c_\epsilon {H_{k-1}}/{H_k^{(k-1)/k}}+\mu H_{k}^{1/k}\right)dM=0.
  \]
  As before, subtracting these equalities, we have
  \[
   \int_M\left(c_\epsilon({H_k^{(k-1)/k}-H_{k-1}})/{H_k^{(k-1)/k}}+\mu (H-H_k^{1/k})\right)dM=0.
  \]

  Again, by  \eqref{eq-hk},  each summand
  in the integrand is non positive. So,
  \[
  H_k^{(k-1)/k}-H_{k-1}=H-H_k^{1/k}=0,
  \]
  giving that $f$ is totally umbilical. This finishes the proof.
\end{proof}

Alternate proofs of Hilbert-Liebmann Theorem in $3$-space forms can be found in
\cite{aledoetal, galvez}.

\section{Convex hypersurfaces of  Riemannian manifolds}

We  continue the considerations of the previous section by presenting  rigidity results
for  locally convex hypersurfaces in Riemannian manifolds of (possibly)
nonconstant sectional curvature. We start with the following
Hadamard-type theorem  due to S. Alexander. (Recall that
a Riemannian manifold $\overbar M^{n+1}$ is called a \emph{Hadamard manifold}
if it is complete, simply connected, and has  nonnegative sectional curvature.)

\begin{theorem}[Alexander \cite{alexander}] \label{th-alexander}
Let $\overbar M^{n+1}$ be a Hadamard manifold. Suppose that $f:M^n\rightarrow\overbar M^{n+1}$  is an
oriented, compact, connected and locally convex hypersurface. Then,  $M$ is diffeomorphic to $\s^n,$
$f$ is an embedding, and $f(M)$ is the boundary of a  convex body in $\overbar M^{n+1}.$
\end{theorem}

Let us outline the elegant proof  Alexander provided for her theorem.

Since $M$ is compact, there exists an open geodesic ball $B$ containing
$f(M)$ which is totally convex, for $\overbar M^{n+1}$ is a Hadamard manifold. Then, setting
$\xi\in TM^\perp$ for the outward unit normal to $f,$
for each $x\in M,$ the geodesic $\gamma_x$ of $\overbar M^{n+1}$ issuing from $x$ with velocity
$\xi(x)$ intersects $\partial B$ transversely at a point $P(x),$ defining a map $P:M\rightarrow\partial B.$
By an extension of the Rauch comparison theorem, due to Warner \cite{warner}, $M$ has no focal points on $\gamma_x,$
so that $P$ is, in fact, a diffeomorphism. In particular, $M$ is diffeomorphic to $\s^n.$

Now, assuming that $f$ is not an embedding, consider the one-parameter family of
parallel hypersurfaces $f_t:M^n\rightarrow\overbar M^{n+1}$ given by
\[
f_t(x)=\exp_{\overbar M}(f(x),t\xi(x)), \,\,\, t\in[0, +\infty).
\]

The local convexity of $f$ and the Bishop Theorem imply that the second fundamental
form $\alpha_f$ is negative semi-definite in the outward normal direction $\xi.$
Also,  a direct computation gives that the principal curvatures of $f_t$ are decreasing
functions of $t,$  so that, for $t>0,$  the second fundamental form of $f_t$ is
negative semi-definite as well.

\begin{figure}[htbp]
\begin{center}
\includegraphics{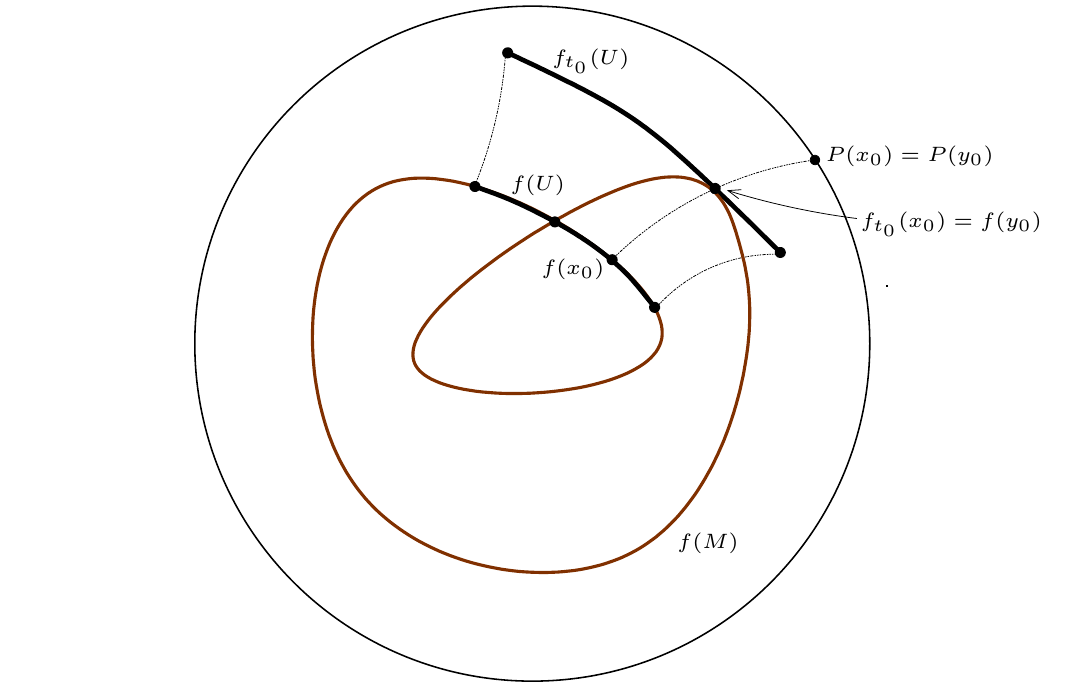}
\caption{\small Proof of Alexander Theorem.}
\label{fig-alexander}
\end{center}
\end{figure}

Now, for a suitable $x\in M$ satisfying  $f(x)=f(y),$ $x\ne y\in M,$ there exists an open neighborhood
$U\subset M$ of $x$ in $M$  such that $f_t(U)$ intersects $f(M)$ for small $t$, and is disjoint from
$f(M)$ for a sufficiently large $t.$ Hence, for some $t_0>0,$ $f_{t_0}(U)$ is tangent to $f(M)$ at a
point $f_{t_0}(x_0)=f(y_0),$ $x_0, y_0\in M$ (Fig. \ref{fig-alexander}). Moreover, the local convexity of
$f_t$ and $f_{t_0}$ gives that their outward unit normal vectors at  $x_0$ and $y_0$ coincide.
In particular, $P(x_0)=P(y_0),$ which contradicts the bijectivity of $P.$ Thus, $f$ is an embedding and
$f(M)$ is the boundary of  a convex body of $\overbar M$ (cf. Remark \ref{rem-convexbody}), which concludes the proof.

In what concerns rigidity of convex surfaces of homogeneous $3$-manifolds, one of the
major results obtained was the following Hadamard-Stoker type theorem by
J. M. Espinar, J. A. G\'alvez, and H. Rosenberg. (We keep their terminology, considering
surfaces as subsets rather than isometric immersions.)

\begin{theorem}[Espinar -- G\'alvez -- Rosenberg \cite{esp-gal-rosen}] \label{th-esp-gal-rosen}
Let $\Sigma$ be a complete connected immersed surface with positive Gaussian curvature in
$\h^2\times\R.$ Then, $\Sigma$ is properly embedded and bounds a convex body in
$\h^2\times\R.$ Moreover, $\Sigma$ is homeomorphic to $\s^2$ or $\R^2.$
In the latter case, $\Sigma$ is a graph over a convex domain of $\h^2\times\R$  or
$\Sigma$ has a simple end.
\end{theorem}

By $\Sigma$ having a simple end means that it has the following  properties:

\begin{itemize}
  \item The asymptotic boundary of the vertical projection  of $\Sigma$
  over $\h^2$ is a singleton $\{p_\infty\}.$
  \item Given a complete geodesic $\gamma$ in $\h^2$ whose ``endpoints at infinity'' are distinct from
  $p_\infty,$ the intersection of the \emph{vertical plane} $\Gamma:=\gamma\times\R$ with $\Sigma$ is either
  empty or compact.
\end{itemize}

An explicit parametrization of a surface $\Sigma$ in $\h^2\times\R$ with positive Gaussian
curvature and one simple end was provided in \cite[Proposition 4.1]{esp-gal-rosen}. Such
a $\Sigma$ is foliated by horizontal horocycles whose vertical projections on $\h^2$ have the
same center $p_\infty$ in the asymptotic boundary $\partial_\infty\h^2$ of the hyperbolic
plane $\h^2$ (Fig. \ref{fig-simpleend}).

\begin{figure}[htbp]
\begin{center}
\includegraphics{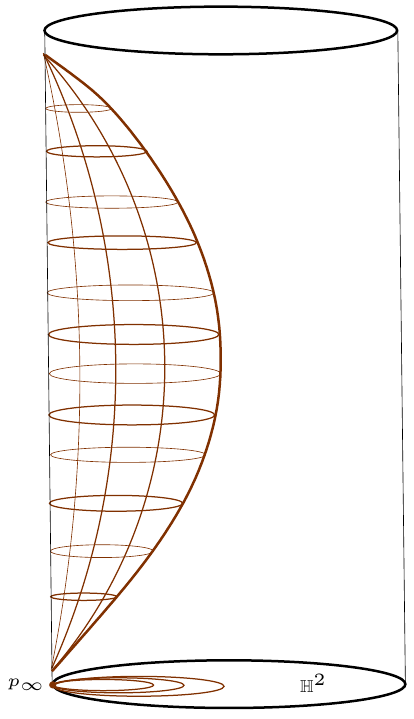}
\caption{\small A strictly convex surface in $\h^2\times\R$ with one simple end.}
\label{fig-simpleend}
\end{center}
\end{figure}

The proof of Theorem \ref{th-esp-gal-rosen} given in \cite{esp-gal-rosen} is divided into two parts.
First, it is assumed that $\Sigma$ has no vertical points, meaning that
there is no $p\in\Sigma$ with  $T_p\Sigma$ parallel to the vertical direction $\partial_t.$
In this case, the intersection of a vertical plane $\Gamma=\gamma\times\R$ with  $\Sigma$ is a
curve which is neither compact nor self-intersecting, since in either of these events $\Sigma$ would
necessarily have a vertical point. In addition, the strict convexity of $\Sigma$
gives that any such embedded curve is strictly convex, that is, has positive geodesic curvature.

\begin{figure}[htbp]
\begin{center}
\includegraphics{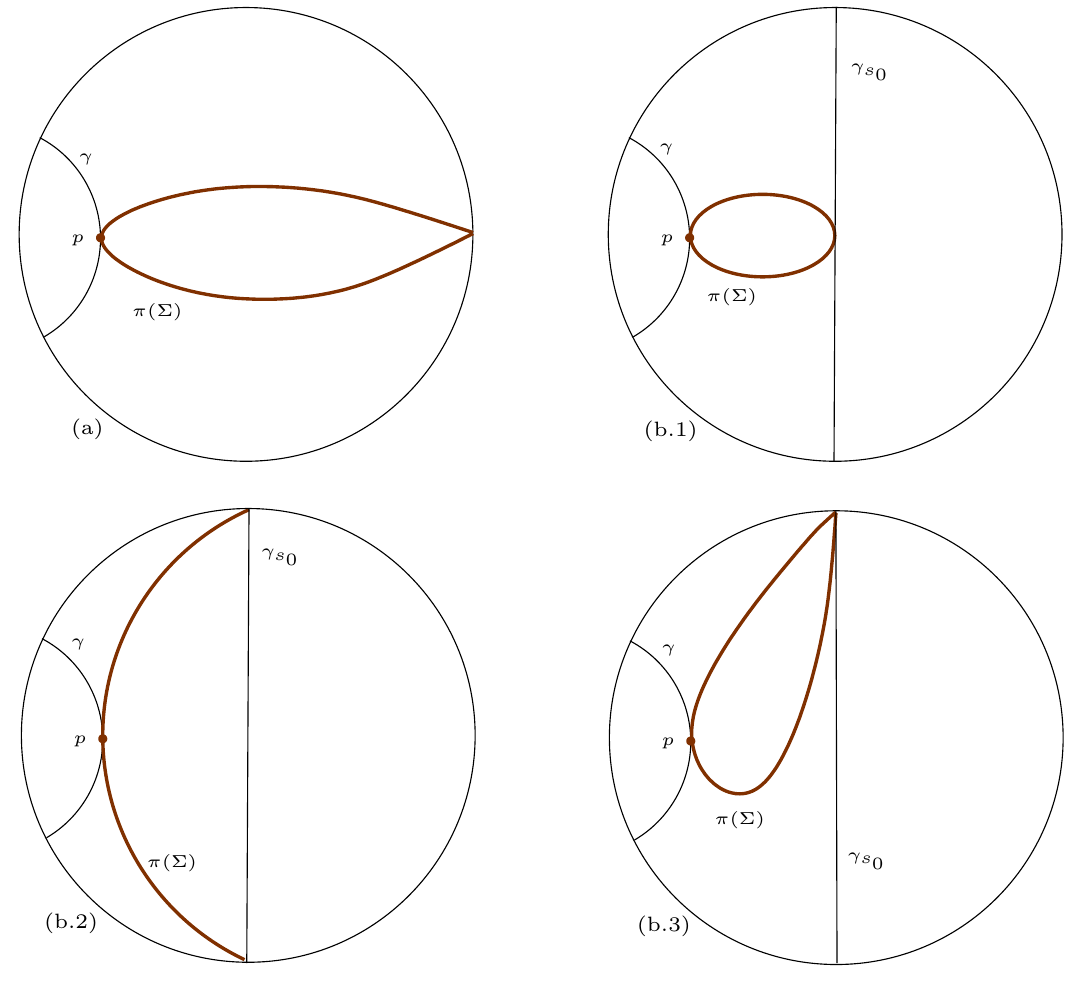}
\caption{\small The projection $\pi(\Sigma)$ in the four cases of the proof of Theorem \ref{th-esp-gal-rosen}.}
\label{fig-esp-gal-rosenProof}
\end{center}
\end{figure}

Now, fix  a geodesic $\gamma^\perp=\gamma^\perp(s)$ of $\h^2,$ orthogonal to $\gamma$ at a point
$q\in\gamma\cap\gamma^\perp,$
and a family $\gamma_s$ of parallel geodesics orthogonal to $\gamma^\perp$ with
$\gamma_0=\gamma.$ Setting $\Gamma_s=\gamma_s\times\R,$ one has that the  intersections
$\Sigma\cap\Gamma_s$ are either empty or embedded vertical graphs over $\gamma_s.$ From this,
one easily concludes that $\Sigma$ is a graph over a convex open set of $\h^2.$ (Notice that
this graph is entire if $\Gamma_s\cap\Sigma$ is nonempty for all $s\in\R.$)

In the second part of the proof, it is assumed that there exists a
vertical point $p\in\Sigma.$ Under this assumption, writing $\Gamma$ for the vertical plane tangent to
$\Sigma$ at $p,$ and $\Gamma_s$ for the  family of parallel vertical planes with
$\Gamma_0=\Gamma,$ one has that
$C(s):=\Gamma_s\cap\Sigma$ is nonempty, embedded and homeomorphic to $\s^1$ for $s>0$ sufficiently
small. Then, a rather delicate reasoning via Morse Theory leads to the conclusion that one of
the following  possibilities occurs as  $s\rightarrow+\infty$ (see Fig. \ref{fig-esp-gal-rosenProof}):

\begin{itemize}
  \item[a)] $C(s)$ is   homeomorphic to $\s^1$ for all $s>0$ --- in which case
  $\Sigma$ is homeomorphic to $\R^2$ and has a simple end.  
  \item[b)] $C(s)$ is nonempty and homeomorphic to $\s^1$ for all $s\in (0,s_0),$ and is empty for $s>s_0$.
  Then, as $s\rightarrow s_0,$ one of the following occurs:
  \subitem b.1)
   The curves $C(s)$ converge to a point $p\in\Sigma$   --- in which case $\Sigma$ is homeomorphic to $\s^2.$
  \subitem b.2) The curves $C(s)$ converge to a curve in the asymptotic boundary of $\h^2\times\R$ ---  in which case
  $\Sigma$ is homeomorphic to $\R^2$.
  \subitem b.3) The vertical projections $\pi(C(s))$ of  $C(s)$ converge to a point
  in $\partial_\infty\h^2$   --- in which case, as in (a),
  $\Sigma$ is homeomorphic to $\R^2$ and has a simple end.
  \end{itemize}

Finally, since $\Sigma$ is embedded and locally strictly convex (and $\h^2\times\R$ is a Hadamard manifold),
it is necessarily the boundary of a convex body  in $\h^2\times\R$ (cf. Remark \ref{rem-convexbody}).

The main results in \cite{esp-gal-rosen} also include horizontal and vertical height
estimates for compact graphs of constant
Gaussian curvature  with boundary in  vertical and horizontal planes, respectively.
These estimates allow one to prove that a complete
surface $\Sigma$ in $\h^2\times\R$ with positive constant Gaussian curvature cannot
be homeomorphic to $\R^2.$ This, together with Theorem \ref{th-esp-gal-rosen}, gives that
such a $\Sigma$ is embedded and homeomorphic to $\s^2.$ Then, applying Alexandrov reflections on
$\Sigma$ with respect to vertical hyperplanes $\Gamma=\gamma\times\R,$ one easily concludes that
$\Sigma$ must be a rotational sphere. Therefore, the following extension of the Hilbert-Liebmann Theorem
for complete (rather than compact!) surfaces of $\h^2\times\R$ holds.

\begin{theorem}[Espinar -- G\'alvez -- Rosenberg \cite{esp-gal-rosen}] \label{th-esp-gal-rosen02}
A complete surface of $\h^2\times\R$ with positive constant Gaussian curvature is a rotational sphere.
\end{theorem}

Regarding the above theorem, it should be mentioned that, in \cite{esp-gal-rosen}, it was proved that there exists a
unique (up to ambient isometries) rotational sphere with constant Gaussian curvature
$K$ in $\h^2\times\R$ for any $K>0.$ Such a surface  also exists in $\s^2\times\R,$ as shown by
Cheng and Rosenberg \cite{cheng-rosenberg}. Based on this fact, and by means of Alexandrov reflections,
Espinar and Rosenberg \cite{esp-rosen}  proved that Theorem \ref{th-esp-gal-rosen02} is valid in $\s^2\times\R$ as well.
Nevertheless, a proof free of Alexandrov reflections (for both cases) was also provided in \cite{esp-gal-rosen}.

The method employed in the proof of Theorem \ref{th-esp-gal-rosen} can be adapted to
establish Hadamard--Stoker type theorems in other contexts. Relying on this fact,
I. Oliveira and S. Schweitzer extended Theorem \ref{th-esp-gal-rosen} to $\h^n\times\R,$ $n\ge 2.$
In this higher dimensional setting, Espinar and Rosenberg \cite{esp-rosen} managed to prove a Hadamard--Stoker type
theorem for properly immersed, locally strictly convex hypersurfaces of $M^n\times\R,$ where $M^n$ is an arbitrary
compact Riemannian manifold with (1/4)-pinched sectional curvature.
Finally, a Hadamard--Stoker type theorem was  obtained by
Espinar and Oliveira \cite{esp-oliv} for certain locally strictly convex surfaces
immersed in  Killing submersions over  Hadamard surfaces.

Inspired by these works, we obtained a Hadamard--Stoker type theorem for immersed locally
strictly convex hypersurfaces of $\overbar M^n\times\R,$ $n\ge 3,$ where $\overbar M^n$  is either
a Hadamard manifold or the unit sphere $\s^n.$ Before stating it, let us  recall that
the height function of a hypersurface  $f:M^n\rightarrow\overbar M^n\times\R$ is the restriction
to $f(M)$ of the projection $\pi_\R$ of  $\overbar M^n\times\R$ on its second factor $\R.$

\begin{theorem}[de Lima \cite{delima}] \label{th-delima}
For $n\ge 3,$ let  $f:M^n\rightarrow\overbar M^n\times\R$  be a
complete connected oriented hypersurface
with positive definite second fundamental form, where $\overbar M^n$  is either
a Hadamard manifold or the unit sphere $\s^n.$ Then, if the  height function  of  $f$
has a critical point, $f$ is a proper embedding  and $M$ is either homeomorphic  to $\s^n$ or $\R^n.$
In particular, $f(M)$  bounds a convex body of $\overbar M^n\times\R$
in the case $\overbar M$ is a Hadamard manifold.
\end{theorem}

For the proof of Theorem \ref{th-delima}, 
we apply Morse Theory together with do Carmo-Warner and Alexander Theorems
to  show that, under the given conditions,
the height function of $f$  has either  a unique critical point,
and then $M$ is homeomorphic to $\R^n,$
or precisely two critical points, and then  $M$ is  homeomorphic to $\s^n.$
In both cases, $f$ is proved to be a proper
embedding.  (Our argument is based on the following fact:
Any  connected component of a transversal intersection $\Sigma_t:=f(M)\cap(\overbar M^n\times\{t\})$
is  a hypersurface of the horizontal hyperplane $\overbar M^n\times\{t\}.$ In addition, such a hypersurface
has positive-definite second fundamental form, since that holds for $f.$ Hence, if $\Sigma_t$ is compact, it is embedded
and homeomorphic to $\s^{n-1},$ by do Carmo-Warner and Alexander Theorems.)

We also considered the dual case of Theorem \ref{th-delima}, as stated below,
in which the height function of the hypersurface has no critical points.

\begin{theorem}[de Lima \cite{delima}]  \label{th-delima02}
For $n\ge 3,$ let $f:M^n\rightarrow\overbar M^n\times\R$  be a proper,
connected, oriented, and cylindrically bounded  hypersurface with positive semi-definite second
fundamental form, where $\overbar M^n$ is either a Hadamard manifold
or the sphere $\s^n.$
Then, $f$ is an embedding, and
$f(M)=\Sigma\times\R,$ where $\Sigma\subset\overbar M^n\times\{0\}$ is a submanifold
homeomorphic to $\s^{n-1}$ which bounds  a convex body in $\overbar M^n\times\{0\}.$
\end{theorem}

We recall that a hypersurface $f:M^n\rightarrow\overbar M^n\times\R$ is said to
be \emph{cylindrically bounded} if there exists a
geodesic ball $B\subset\overbar M^n\times\R$ such that $f(M)\subset B\times\R.$
Notice that, if $\overbar M^n\times\R$ is compact, any hypersurface
$f:M^n\rightarrow\overbar M^n\times\R$ is cylindrically bounded.

Let us outline the proof of Theorem \ref{th-delima02}.

First, observe that
the height function of $f$ is unbounded, since it has no critical points and
$f$ is proper. Then, proceeding as in the proof of Theorem \ref{th-delima}, we conclude
from do Carmo--Warner and Alexander Theorems that $f(M)$ is foliated by
embedded locally convex topological $(n-1)$-spheres, so that
$M$  is homeomorphic to $\s^{n-1}\times\R,$ and $f$ is an embedding.

The local convexity of $f$ implies that $f(M)$ is the boundary of a convex set
of $\overbar M^n\times\R$ in the case $\overbar M^{n}$ is a Hadamard manifold. This, together with
the cylindrical boundedness of $f,$ easily implies that $f(M)=\Sigma\times\R,$ as stated.

In the spherical case, we have to consider the Gauss formula for hypersurfaces
 $f:M\rightarrow\s^n\times\R:$
\[
K_M(X,Y)=\det A|_{{\rm span}\{X,Y\}}+(1-\|T(X,Y)\|^2), \,\,\, X, Y\in TM,
\]
where $A$ is the shape operator of $f$ and $T(X,Y)$ is the orthogonal projection of
the gradient $T$ of its height function on ${\rm span}\{X,Y\}$  (see \cite{daniel}).
We add that $T$ itself is  the orthogonal projection of the vertical unit field $\partial_t$
on the tangent bundle $TM.$

From the above Gauss equation,
the sectional curvature $K_M$ of $M$ is nonnegative,  for $\det A|_{{\rm span}\{X,Y\}}\ge 0$ (by the
positive semi-definiteness of the second fundamental form of $f$),
and $\|T(X,Y)\|\le 1.$  Since $M$ is noncompact, for all $x\in M,$
there exist orthonormal vectors $X, Y\in T_xM$ satisfying $K(X,Y)=0.$  Otherwise,
the celebrated  Soul Theorem, by T. Perelman \cite{perelman},
would give that $M$ is homeomorphic to $\R^n.$
By applying Gauss equation to $(X,Y),$  one gets
\[
\|T(X,Y)\|^2=1+\det A|_{{\rm span}\{X,Y\}}\ge 1.
\]
Therefore, $\|T(X,Y)\|=1,$ which gives
$T=\partial_t,$  and then  that  $f(M)$ is a vertical cylinder
over a convex hypersurface $\Sigma$ of $\s^n.$

Once we have Theorem \ref{th-delima}, we can proceed as in the
proof of Theorem \ref{th-esp-gal-rosen02} to establish a Jellett--Liebmann type theorem
in the products $\mathbb{Q}_\epsilon^{n}\times\R, \,\epsilon\ne 0.$ Namely, we  perform
Alexandrov reflections on a given compact hypersurface $f:M\rightarrow\mathbb{Q}_\epsilon^{n}\times\R$
with respect to vertical hyperplanes $\Gamma=\Gamma_0\times\R,$ where
$\Gamma_0\subset\mathbb Q_\epsilon^n$ is a totally geodesic hypersurface  of $Q_\epsilon^n$
(for $\epsilon=1,$ $\Gamma_0$ is assumed to be in an open hemisphere of $\s^n$). The statement is as follows.

\begin{theorem}[de Lima \cite{delima}] \label{th-delima03}
For $n\ge 3$ and $\epsilon\in\{-1,1\},$  any compact, connected, and  locally strictly convex hypersurface
$f:M^n\rightarrow\mathbb{Q}_\epsilon^{n}\times\R$ with constant mean curvature
is congruent to an embedded rotational sphere. 
\end{theorem}

Concerning Theorem \ref{th-delima03}, we add
that Hsiang--Hsiang \cite{hsiang} (resp. R. Pedrosa \cite{pedrosa}) constructed
embedded and strictly convex rotational spheres in $\h^n\times\R$ (resp. $\s^n\times\R$)
of constant mean curvature.

Theorem \ref{th-delima} can also be used to establish a Hilbert--Liebmann type theorem
in $\mathbb{Q}_\epsilon^{n}\times\R$ for hypersurfaces of constant sectional curvature.
For $n\ge 3,$ such hypersurfaces were constructed and classified by F. Manfio and R. Tojeiro
in \cite{manfio-tojeiro}. The following result, for which we provided a distinct proof,
is part of their classification results.

\begin{theorem}[de Lima \cite{delima}, Manfio--Tojeiro \cite{manfio-tojeiro}]  \label{th-delima04}
Let $M^n_c$ be a complete, connected and
orientable $n(\ge 3)$-dimensional Riemannian manifold with constant sectional curvature $c.$
Given an isometric immersion
$f:M_c^n\rightarrow\mathbb{Q}_\epsilon^{n}\times\R, \,\epsilon\in\{-1,1\},$
assume that $c>(1+\epsilon)/2.$  Then, $f$ is congruent
to an embedded rotational sphere.
\end{theorem}

In our proof of Theorem \ref{th-delima04}, we start by noticing that the condition
on $c,$ together with Myers Theorem and Gauss equation, gives that $M$ is compact and has
positive definite second fundamental form. Hence, from Theorem \ref{th-delima},
$M$ is a topological sphere of constant curvature $c$ and $f$ is an embedding.
Then, by using the fact that the gradient  of the height function of the hypersurface $f$ (when nonzero)
is one of its principal directions (as proved in \cite{manfio-tojeiro}), we managed to prove that
$f(M)$ is the connected sum of two rotational embedded hemispheres with a common axis, showing that
$f(M)$ is indeed a rotational embedded sphere.

We conclude our considerations by presenting a result due to
H. Rosenberg and R. Tribuzy \cite{rosenberg-tribuzy} on rigidity of  convex
surfaces of the  homogeneous $3$-manifolds known as $\mathbb{E}(k,\tau)$ \emph{spaces.}

Given $k, \tau\in\R$ with $k-4\tau^2\ne 0,$ one denotes by
$\mathbb{E}(k,\tau)$ the  total space of a Riemannian submersion over the simply connected
two-dimensional space form  of curvature $k$ with bundle curvature $\tau.$
The unit tangent field to the fiber, which is a Killing field, is denoted by $\xi.$
The $\mathbb{E}(k,\tau)$ spaces include
the products $\h^2\times\R$ and $\s^2\times\R$, the Heisenberg space ${\rm Nil}_3$,
the Berger spheres, and  the universal cover of the special linear group ${\rm SL}_2(\R)$.

As defined in \cite{rosenberg-tribuzy}, an oriented surface $f\colon M\rightarrow\mathbb{E}(k,\tau)$
is said to be \emph{strictly convex} if any of its principal curvatures  is at least $\tau.$
With this terminology, the Rosenberg--Tribuzy Theorem reads as follows.

\begin{theorem}[Rosenberg -- Tribuzy \cite{rosenberg-tribuzy}]
  Let $f(t)\colon M\rightarrow\mathbb{E}(k,\tau)$ be a smooth one-parameter
  family of isometric immersions with $f(0)=f.$ Suppose
  $f$  is strictly convex, $K(f_t(x)) = K(f(x))$ for $x\in M$ and all $t,$ and
  $H(f_t(x))= H(f(x))$ at a non-horizontal point $x$ of $M$. Then,
  there are isometries $h(t)\colon\mathbb{E}(k,\tau)\rightarrow\mathbb{E}(k,\tau)$
  such that $h(t)f(t)=f.$
\end{theorem}

In the above statement, $K(f(x))$ and $H(f(x))$ denote the Gaussian and mean curvatures,
respectively,  of  the immersion $f$ at a point $x\in M.$ Also,
a point $x\in M$ is called \emph{horizontal} (resp. \emph{vertical})
if $\xi$ is orthogonal (resp. parallel) to  $T_xM.$

The idea of the proof is to  consider first the unit field
\[
e_1:=\frac{P(\xi)}{\|P(\xi)\|}
\]
on the open set $U\subset M$ of non horizontal points of $f,$ where
$P$ denotes the orthogonal projection on the tangent bundle $TU.$ Then, on $U,$
there is a well defined differentiable angle function $\theta$ such that
\[
\xi=(\cos\theta) e_1+(\sin\theta) N,
\]
where $N$ is the unit normal to $f.$ Moreover, it is  shown that, due to the convexity of $f,$
for any vertical point $x\in M,$ $\theta$ is a submersion in a neighborhood of $x.$
This allows one to choose a unit field $v\in TU$ orthogonal
to the gradient of $\theta,$ and then define a second angle function $\phi$ on $U$ as
\[
v=(\cos\phi)e_1+(\sin\phi)e_2,
\]
where $e_2=Je_1$ and $J$ is the positive $\pi/2$-rotation.

In this setting,
a computation (in which the convexity of $f$ plays a fundamental role) shows that
$\phi$ is the solution of an ODE which involves the function $\theta.$
From uniqueness of solutions of ODE's satisfying initial conditions, and the hypotheses on
the curvature functions $H$ e $K,$ one concludes that the functions $\phi$ and $\theta$
are the same for all immersions $f(t).$

It is also proved from the convexity of $f$  that the horizontal points of any $f(t)$ are all isolated
and have index one. From this, and
the above considerations, one has that all immersions $f(t)$ satisfy the hypotheses of
the fundamental theorem for immersions in homogeneous $3$-manifolds, by B. Daniel \cite{daniel2},
proving that they are all congruent to each other, as asserted.

\section*{Acknowledgements}
The author is indebt to both anonymous referees for their corrections and valuable suggestions.
They have considerably improved the presentation of the paper.

\end{document}